\documentclass[11pt]{article}
\usepackage{a4wide}
\usepackage{amsmath}
\usepackage{amsfonts}
\usepackage{amsthm}
\usepackage[utf8]{inputenc}
\usepackage{multirow}
\newcommand{\N}{{\mathbb N}}
\newcommand{\Z}{{\mathbb Z}}
\newcommand{\Q}{{\mathbb Q}}
\newcommand{\R}{{\mathbb R}}
\newcommand{\semi}{{\mathbb o}}
\newcommand{\lie}{{\mathfrak g}}

\newcommand{\Aut}{{\rm Aut}}

\newcommand{\Aff}{{\rm Aff}}
\newcommand{\aff}{{\rm aff}}
\newcommand{\Endo}{{\rm Endo}}
\newcommand{\dd}{\delta}
\newcommand{\D}{{\mathfrak D}}
\newcommand{\A}{{\mathfrak A}}
\newcommand{\B}{{\mathfrak B}}
\newcommand{\GL}{{\rm GL}}
\newtheorem{theorem}{Theorem}[section]
\newtheorem{proposition}[theorem]{Proposition}
\newtheorem{lemma}[theorem]{Lemma}
\newtheorem{corrolary}[theorem]{Corrolary}
\newtheorem{definition}[theorem]{Definition}
\newtheorem*{remark}{Remark}

\newcommand{\ds}{\displaystyle}
\newcommand{\ra}{\rightarrow}

\begin{document}
\title{\bf Nielsen zeta functions for maps on infra-nilmanifolds are rational}
\author{Karel Dekimpe and Gert-Jan Dugardein\\
KULeuven Kulak, E.\ Sabbelaan 53, B-8500 Kortrijk}
\date{\today}
\maketitle

\begin{abstract}
In this paper we will show that for any map $f$ on an infra-nilmanifold, the Nielsen number 
$N(f)$ of this map is either equal to $|L(f)|$, where $L(f)$ is the Lefschetz number of that map, or equal to the expression $|L(f)-L(f_+)|$, where $f_+$ is a lift of $f$ to a 2-fold covering 
of that infra-nilmanifold. By exploiting the exact nature of this relationship 
for all powers of $f$, we prove that 
the Nielsen dynamical zeta function for a map on an infra-nilmanifold is always a rational function. 
\end{abstract}
\section{Introduction}
Let $X$ be a compact polyhedron and $f:X\ra X$ be a self-map. We can attach
 two different numbers to this map $f$, each one providing information on the number of fixed points
 of $f$. The first one is the Lefschetz number $L(f)$ which is defined as
 \[ L(f)= \sum_{i=0}^{{\rm dim}\;X} (-1)^i {\rm Tr} \left( f_{\ast,i} : H_i(X,\Q) \ra H_i(X,\Q)\right).\]
The main result about the Lefschetz number is that any map homotopic to $f$ has at least one fixed point, when $L(f)\neq 0$. The Nielsen number $N(f)$, on the other hand, is harder to define. It will always be a nonnegative integer which, in general, will be a lot harder to compute than $L(f)$. It gives more information about the self-map $f$, though, since any map homotopic to $f$ will have at least $N(f)$ fixed points. We refer the reader to \cite{brow71-1} for more information on both the Lefschetz and the Nielsen number.

\medskip
In discrete dynamical systems, both numbers are used to define a so-called dynamical zeta 
function (\cite{fels00-2}). The Lefschetz zeta function of $f$, which was introduced by S.~Smale (\cite{smal67-1}), is given by 
\[ L_f(z)=\exp\left( \sum_{k=1}^{+\infty}\frac{L(f^k)}{k}z^k\right).\]
Analogously, A.~Fel'shtyn \cite{fels88-1,fp85-1} 
introduced the Nielsen zeta function, which  is given by 
\[ N_f(z)=\exp\left( \sum_{k=1}^{+\infty}\frac{N(f^k)}{k}z^k\right)\]

In \cite{smal67-1} the following theorem was obtained (although Smale only considered diffeomorphisms on compact manifolds in his paper, the Lefschetz dynamical zeta function is 
defined for all maps on all compact polyhedra and his result also holds for all of these maps):
\begin{theorem}\label{Smale}
The Lefschetz zeta function for self-maps on compact polyhedra is rational.
\end{theorem}

It has been shown that the Nielsen zeta function has a positive radius of convergence (\cite{fp85-1}), 
but unlike the Lefschetz zeta function, the Nielsen zeta function does not have to be rational in general.
The question of when the Nielsen zeta function is rational has been studied in several papers, e.g.\ \cite{fels00-2,fels00-1,fh99-1,li94-1,fp85-1,wong01-1}.

\medskip

In this paper we treat this problem for maps on infra-nilmanifolds. As it is known that the Nielsen zeta function on nilmanifolds is always rational it was very natural to ask the same question for 
infra-nilmanifolds.   
Until now, only a very partial result on this problem can be found in \cite[Theorem 4]{wong01-1}
where a (rather technical) condition is given under which the rationality of the 
Nielsen zeta function is guaranteed.

\medskip

The main result of this paper is that the Nielsen zeta function of any map on any infra-nilmanifold is a rational function (Corollary~\ref{mainresult}).
In order to obtain this result we show that for any map $f$ on an infra--nilmanifold, we either 
have $N(f)=|L(f)|$ or $N(f)=|L(f) - L(f_+)|$ where $f_+$ is a lift of $f$ to a 2-fold covering of 
the given infra-nilmanifold $f$. 
In fact, it was already known that many maps on infra-nilmanifolds satisfy the Anosov relation 
(\cite{anos85-1,ddm04-1,ddm05-1,ddp11-1,fh86-1,km95-1}) and these are exactly the maps for which the first condition holds. The second condition now clearly shows what happens for those maps that do not satisfy the Anosov relation.

\medskip

Using these relations we are able to describe the Nielsen zeta function of 
$f$ in terms of the Lefschetz zeta function of $f$ and $f_+$ from which the rationality then easily follows using Smale's result.

\section{Infra-nilmanifolds}
Let us now describe the class of infra-nilmanifolds in some detail. Any infra-nilmanifold is modeled on a connected and simply connected nilpotent Lie group $G$. Given such a Lie group $G$,
we consider its affine group which is the semi-direct product $\Aff(G)= G\semi \Aut(G)$. The group 
$\Aff(G)$ acts on $G$ in the following way:
\[ \forall (g,\alpha)\in \Aff(G),\, \forall h \in G: \;\;^{(g,\alpha)}h= g \alpha(h).\]
Note that when $G=\R^n$, $\Aff(\R^n)$ is the usual affine group, acting in the usual way on 
$\R^n$. Note also that since $\R^n$ is simply connected and abelian (hence a fortiori nilpotent), this case is included in our discussion. We will use $p:\Aff(G)=G\semi \Aut(G) \ra \Aut(G)$ to denote the natural projection on the second factor.  

\begin{definition} Let $G$ be a connected and simply connected nilpotent Lie group. 
A subgroup $\Gamma \subseteq \Aff(G)$ is called an almost--crystallographic group (modeled on $G$) if and only if $p(\Gamma)$ is finite and $\Gamma\cap G$ is a uniform and discrete subgroup of $G$. The finite group $F=p(\Gamma)$ is called the holonomy group of $\Gamma$.
\end{definition}

Being a subgroup of $\Aff(G)$, any almost--crystallographic group $\Gamma$ acts on $G$. This action is properly discontinuous and cocompact. In case $\Gamma$ is torsion-free, this action is free and the quotient space $\Gamma\backslash G$ is a manifold (with universal covering space $G$ and fundamental group $\Gamma$). These manifolds are exactly the ones called ``infra-nilmanifolds''.
\begin{definition}
A torsion-free almost--crystallographic group $\Gamma\subseteq \Aff(G) $ 
is called an almost--Bieberbach group, and the corresponding manifold $\Gamma\backslash G$ is said to be an infra--nilmanifold (modeled on $G$). When $\Gamma \subseteq G$, i.e.\ when the holonomy group $p(\Gamma)$ is trivial, the corresponding manifold $\Gamma\backslash G$ is a 
nilmanifold.  
\end{definition} 

For any almost--Bieberbach group $\Gamma$ modeled on a Lie group $G$, we have that 
$N=G\cap \Gamma$ is of finite index in $\Gamma$ and hence, the infra-nilmanifold $\Gamma\backslash G$ is finitely covered by the nilmanifold $N\backslash G$, explaining the name ``{\em infra}''--nilmanifold.  In case $G=\R^n$, we talk about crystallographic groups and Bieberbach groups. In this case, the infra--nilmanifolds are the compact flat manifolds and any such manifold is covered by a torus $T^n$ (because $N\cong \Z^n$).

\medskip

In order to study the Nielsen theory of an infra--nilmanifold, we need to understand all maps on such a manifold up to homotopy. A complete description of these maps, is given by the work of K.B.\ Lee \cite{lee95-2}. Here we formulate the results for maps between two infra-nilmanifolds modeled on the same nilpotent Lie group, but this result has a straightforward extension to infra-nilmanifolds modeled on different Lie groups.  In order to formulate this result, we extend the affine group of $G$ to the semigroup of affine 
endomorphisms $\aff(G)$ of $G$. Here $\aff(G)=G\semi \Endo(G)$, where $\Endo(G)$ is the semigroup of endomorphisms of $G$. An element of $\aff(G)$ is a pair $(\dd,\D)$, where $\dd\in G$ and $\D\in \Endo(G)$. Such an element should be seen as an ``affine map'' on the Lie group $G$
\[ (\dd,\D): \; G \rightarrow G:\; h \mapsto \dd \D(h).\]
Note that in this way $\Aff(G)\subseteq \aff(G)$.

\begin{theorem}[K.B.\ Lee \cite{lee95-2}]
\label{leemaps} Let $G$ be a connected and simply connected nilpotent Lie group and suppose that $\Gamma, \Gamma'\subseteq \Aff(G)$ are two almost-crystallographic groups modeled on $G$. 
Then for any homomorphism $\varphi: \Gamma\rightarrow \Gamma'$ there 
exists an element $  (\dd, \D)\in \aff(G)$ such that 
\[ \forall \gamma \in \Gamma: \; \varphi(\gamma) (\dd, \D) =  (\dd, \D) \gamma.\] 
\end{theorem}

Note that the equality $ \varphi(\gamma) (\dd, \D) =  (\dd, \D) \gamma$ makes sense, because it involves three elements of $\aff(G)$. From this equality one can see that the affine map $(\dd,\D)$ descends to a map 
\[\overline{(\dd,\D)}: \Gamma \backslash G \rightarrow \Gamma' \backslash G: \; \Gamma h \rightarrow \Gamma' \dd \D(h)\]
which exactly induces the morphism $\varphi$ on the level of the fundamental groups.
We will say that $\overline{(\dd,\D)}$ is induced from an affine map.

\medskip

Now, let $f:\Gamma\backslash G\ra \Gamma'\backslash G$  be any map between two infra-nilmanifolds and let $\tilde{f}:G \ra G$ be a lift of $f$. 
Then $\tilde{f}$ induces a morphism $\varphi:\Gamma\ra \Gamma'$ determined by $\varphi(\gamma) \circ \tilde{f} = \tilde{f}\circ \gamma$, 
for all $\gamma\in \Gamma$. From Theorem~\ref{leemaps} it follows that there also exists an affine map $(\dd,\D)\in \aff(G)$ satisfying 
$\varphi(\gamma) \circ (\dd,\D)= (\dd,\D)\circ \gamma$ for all $\gamma\in \Gamma$. Therefore, the induced map $\overline{(\dd,D)} $ and 
$f$ are homotopic.
\begin{definition}
With the notations above, we will say that $(\dd,D)$ is an affine homotopy lift of $f$.
\end{definition}

As the Nielsen and Lefschetz numbers are homotopy invariants, it will suffice to study maps induced by an affine map. 
For those maps there are very convenient formulae to compute the Lefschetz and the Nielsen numbers. Let 
us fix an infra-nilmanifold $\Gamma\backslash G$ which is determined by an almost--Bieberbach group $\Gamma\subseteq \Aff(G)$ and let $F\subseteq \Aut(G)$ be the holonomy group of $\Gamma$. We will denote the Lie algebra of $G$ by $\lie$. Recall that there is an isomorphism 
between $\Aut(G)$ and $\Aut(\lie)$  which associates to each automorphism $\A\in \Aut(G)$ its differential $\A_\ast\in \Aut(\lie)$ at the identity element of $G$. 

\begin{theorem}[J.B.\ Lee and K.B.\ Lee \cite{ll09-1}] \label{LeeForm}Let $\Gamma\subseteq \Aff(G)$ be an almost-Bieberbach group with holonomy group $F\subseteq \Aut(G)$. Let $M=\Gamma\backslash G$ be the associated infra-nilmanifold.  If 
 $f:M\ra M$ is a map with affine homotopy lift $(\dd, \D)$, then  
\[L(f)=\frac{1}{\# F}\sum_{\A \in F}\det(I-\A_\ast\D_\ast)\]
and
\[N(f)=\frac{1}{\# F}\sum_{\A \in F}|\det(I-\A_\ast\D_\ast)|.\]
(Here $I$ is the identity matrix).
\end{theorem}

\section{The holonomy representation and maps}
To any almost--crystallographic group, and hence also to any infra--nilmanifold, we can associate its holonomy representation.
\begin{definition}
Let $\Gamma\subseteq \Aff(G)$ be an almost--crystallographic group modeled on $G$ and with holonomy group $F=p(\Gamma) $. The holonomy representation of $\Gamma$ is the representation 
\[ \rho: F \rightarrow \GL(\lie): \A \mapsto \A_\ast\]
\end{definition}
By choosing a basis of $\lie$, we can identify $\lie$ with $\R^n$ for some $n$ and therefore we can view the holonomy representation $\rho$ as being a 
real representation $\rho: F \rightarrow \GL_n(\R)$. 

\medskip

There is a strong connection between this  holonomy representation and the affine homotopy lift of a map on the corresponding infra--nilmanifold.

\begin{proposition}[See \cite{ddp11-1}]\label{semi-affconj}
Let $\Gamma\subseteq \Aff(G)$ be an almost--Bieberbach group and let $M=\Gamma\backslash G$ be the corresponding infra--nilmanifold. Let $\rho:F \ra \GL(\lie)$ be the associated holonomy representation. If $f:M\ra M$ is a map with affine homotopy lift $(\dd,\D)$, there exists a function $\phi:F \ra F$ such that 
\[ \forall x \in F:\; \rho(\phi(x)) \D_\ast = \D_\ast \rho(x) .\]
\end{proposition}  

It is tempting to believe that the function $\phi$ will be a morphism of groups, however as an example in \cite{ddp11-1} shows, this need not be the case.

\medskip

In the following proposition, which was mainly proved in \cite{ddp11-1}, 
we show how a map $f$ on an infra--nilmanifold induces a 
decomposition of the holonomy representation into two subrepresentations.

\begin{proposition}\label{decomprho}
Let $\rho:F\to \GL_n(\R)$ be a representation of a finite group $F$ and $\phi: F\to F$ be any function. Let $D$ be a linear transformation of $\R^n$ (which we view as an $n\times n$ matrix w.r.t.\  the standard basis). Suppose that $\rho(\phi(x))D=D\rho(x)$ for all $x\in F$. Then we can choose a basis of $\R^n$, such that $\rho=\rho_{\leq 1}\oplus \rho_{> 1}$, for representations $\rho_{\leq 1}:F\to \GL_{n_{\leq 1}}(\R)$ and $\rho_{> 1}:F\to \GL_{n_{> 1}}(\R)$ and such that $D$ can be written in block triangular form$$\left(\begin{array}{cc}
D_{\leq 1} & \ast \\
0 & D_{>1}
\end{array}\right),$$where $D_{\leq 1}$ and $D_{>1}$ only have eigenvalues of modulus $\leq 1$ and $>1$, respectively. 
\end{proposition}

\begin{proof}
The proof of this proposition can be extracted from the more general proof one can find in \cite[page 545]{ddp11-1}. For the reader's convenience, we will recall those steps of the original proof that suffice for the proof of our proposition. One first shows that the generalized eigenspace of $D$ with respect to the eigenvalue 0 is an $F$--subspace of $\R^n$ and so one obtains a decomposition 
$\R^n= V_0 \oplus V_1$ such that $D$ takes up the form 
\[ D= \left(\begin{array}{cc}
D_0 & \ast \\
0 & D_1
\end{array}\right),\]with respect to a basis of $\R^n$ consisting of a basis of $V_0$ complemented with a basis of $V_1$. Note that $D_0$ only has $0$ as an eigenvalue, while $D_1$ only has non-zero eigenvalues.
Also, the representation $\rho$ decomposes as $\rho=\rho_0\oplus \rho_1$.
Then the space $\R^n/V_0\cong V_1$ together with the representation $\rho_1$ and the linear transformation 
$D_1$ is considered and it is shown that this space has a direct decomposition 
$V_1= W_{\leq 1} \oplus W_{>1}$ as $F$--spaces such that $D_1$ is of the form 
\[ D_1= \left(\begin{array}{cc}
D'_{\leq 1} & 0 \\
0 & D'_{>1}
\end{array}\right),\]
where $D'_{\leq 1}$ only has (non-zero) eigenvalues $\lambda$ of modulus $\leq1 $, while 
$D'_{>1}$ only has eigenvalues $\lambda$ of modulus $>1$. The proof now finishes by taking 
$V_{\leq 1}=V_0 \oplus  W_{\leq 1}$ and $V_{>1}= W_{>1}$, and so 
\[ D_{\leq 1}= \left(\begin{array}{cc}
D_0 & \ast \\
0 & D'_{\leq 1}
\end{array}\right)\mbox{ and } D_{\geq 1} = D'_{\geq1}.\]
\end{proof}

Let $M=\Gamma\backslash G$ be an infra--nilmanifold, whose fundamental group is the almost--Bieberbach group $\Gamma\subseteq \Aff(G)$, having  
$F$ as its  holonomy group and $\rho:F\rightarrow \GL(\lie)$ as its holonomy representation. Given a self-map 
$f:M\rightarrow M$ with affine homotopy lift $(\dd,\D)$, Proposition~\ref{semi-affconj} shows that there exists a map $\phi$ such that 
$\rho:F\ra \GL(\lie)$, $\phi$ and $\D_\ast$ satisfy the conditions of Proposition~\ref{decomprho}.
\begin{definition}
In this specific case, we will refer to the decomposition $\rho=\rho_{\leq 1}\oplus \rho_{>1}$ obtained from Proposition~\ref{decomprho} as the decomposition of $\rho$ induced by $\D$.
\end{definition}
As already indicated in the introduction, it is our aim to have a good understanding of the exact relationship between the Nielsen and the Lefschetz number of a given map $f$ on an infra-nilmanifold. From Theorem~\ref{LeeForm} it is clear that the terms $\det(I-\rho(x) \D_\ast)$ and especially their signs (in order to obtain the modulus of these terms) will play a crucial role in understanding this relationship. In the following lemma and proposition, we will therefore deduce how these signs behave.
\begin{lemma}\label{expanding}
Let $\rho:F\to \GL_n(\R)$ be a representation of a finite group $F$ and $\phi: F\to F$ be any function. Let $D$ be a linear transformation of $\R^n$. Suppose that $\rho(\phi(x))D=D\rho(x)$ for all $x\in F$. Suppose that $|\lambda|>1$ for all eigenvalues $\lambda$ of $D$. Then the following statement holds: $$\forall x \in F: \ \det(\rho(x))\det(I-D)\det(I-\rho(x)D)>0.$$ 
\end{lemma}
\begin{proof}
This proof is largely based on the proof of \cite[Theorem 3.2]{ddm05-1}. Choose an arbitrary $x \in F$. We can define a sequence $(x_i)_{i\in \N}$ 
of elements in $F$ by taking $x_1=x$ and such that $x_{i+1}= \phi(x_i)$. 
Since $F$ is finite, this sequence will become periodic from a certain point onwards. By \cite[Lemma 3.1]{ddm05-1}, we know that 
\begin{equation}\label{induct1}\forall i \in \N: \;\det(I-\rho(x_i) D) =\det(I-\rho(x_{i+1}) D).\end{equation}
Also, by the same lemma, there exists an $l\in \N$ and an element $x_j$ in our sequence such that $(\rho(x_j)D)^l=D^l$. 
As every eigenvalue of $D$ has modulus $>1$, we know that every eigenvalue of $\rho(x_j)D$ will also have modulus $>1$. 
Let us call those eigenvalues $\lambda_1,\dots ,\lambda_n$, then 
$$\det(I-\rho(x_j) D)=(1-\lambda_1)\dots(1-\lambda_n).$$
Note that the complex eigenvalues, which always come in conjugate pairs, together with the negative real eigenvalues of $\rho(x_j) D$ can only give a positive 
contribution to this product. So the sign of $\det(I-\rho(x_j)D)$ is completely determined by the parity of the number of positive real eigenvalues. 
Analogously, the sign of $\det(I-D)$ is completely determined by the parity of the number of real positive eigenvalues of $D$.

\medskip

As $\rho(x_j)$ is a real matrix of finite order, we know that $\det(\rho(x_j))$ equals $1$ or $-1$. 
If $\det(\rho(x_j))=1$, then $\det(\rho(x_j)D)=\det(D)$. A fortiori, $\det(\rho(x_j)D)$ and $\det(D)$ have the same sign (and are both non-zero).
Hence, the parity of the number of negative real eigenvalues of $\rho(x_j)D$ and $D$ is the same, and therefore also the 
parity of the number of positive real eigenvalues of both matrices is the same (since complex eigenvalues come in conjugate pairs). It follows that 
in this case $\det(I-D)$ and $\det(I-\rho(x_j)D)$ have the same sign and, hence
\begin{equation}\label{inequality1}
\det(\rho(x_j))\det(I-D)\det(I-\rho(x_j)D)=\det(I-D)\det(I-\rho(x_j)D)>0.
\end{equation}

When $\det(\rho(x_j))=-1$, we deduce in a similar way that the parity of the number of positive eigenvalues of $\rho(x_j)D$ and $D$ is different,
hence $\det(I-D)$ and $\det(I-\rho(x_j)D)$ have an opposite sign and we find 
\begin{equation}\label{inequality2}
\det(\rho(x_j))\det(I-D)\det(I-\rho(x_j)D)=-\det(I-D)\det(I-\rho(x_j)D)>0.
\end{equation}

Note that for every $i \in \N$, we have$$\det(\rho(x_{i+1}))\det(D)=\det(\rho(x_{i+1})D)=\det(D\rho(x_{i}))=\det(D)\det(\rho(x_{i})).$$Because $\det(D)\neq 0$, this means that $$\det(\rho(x_{i+1}))=\det(\rho(x_{i})).$$By using an inductive argument on this expression and on expression \eqref{induct1}, we find that $$\det(\rho(x))=\det(\rho(x_{j})) \textrm { and } \det(I-\rho(x) D)
=\det(I-\rho(x_{j}) D).$$
This, together with the two inequalities (\ref{inequality1}) and (\ref{inequality2}), proves this lemma.  
 \end{proof}
\begin{proposition}\label{Sign}
Let $\Gamma\subseteq \Aff(G)$ be an almost--Bieberbach group and let $M=\Gamma\backslash G$ be the corresponding infra--nilmanifold. 
Let $F$ be the holonomy group of $\Gamma$ and 
$\rho:F\to \GL_n(\lie)$ be the associated  holonomy representation. 
Choose an arbitrary self-map $f:M\to M$ and let $(\dd,\D)\in \aff(G)$ be an affine homotopy lift of $f$.
Let $\rho=\rho_{\leq 1}\oplus \rho_{>1}$ be the decomposition of $\rho$ induced by $\D$.  For every $x \in F$, 
the following statements hold:
\[\det(\rho_{>1}(x))=1\Rightarrow \det(I-\D_\ast)\det(I-\rho(x)\D_\ast)\geq 0\]and
\[ \det(\rho_{>1}(x))=-1\Rightarrow \det(I-\D_\ast)\det(I-\rho(x)\D_\ast)\leq 0.\]
\end{proposition}
\begin{proof}
From Propositions~\ref{semi-affconj} and \ref{decomprho}, we know that there exists a function $\phi:F\to F$, such that 
$\rho(\phi(x)) \D_\ast = \D_\ast \rho(x)$, for all $x\in F$ and there exists a decomposition of $\lie$ into two subspaces
leading to the decomposition $\rho=\rho_{\leq 1}\oplus \rho_{> 1}$, while $\D_\ast$ can be written in block diagonal form 
$$\left(\begin{array}{cc}
D_{\leq 1} & \ast \\
0 & D_{>1}
\end{array}\right),$$
where $D_{\leq 1}$ and $D_{>1}$ only have eigenvalues of modulus $\leq 1$ and $>1$, respectively. 

For every $x\in F$ we have that
$$\det(I-\rho(x)\D_\ast)=\det(I-\rho_{>1}(x)D_{>1})\det(I-\rho_{\leq 1}(x)D_{\leq 1}).$$
Analogously as in the proof of Lemma~\ref{expanding}, we can show that $\rho_{\leq 1}(x)D_{\leq 1}$ only has eigenvalues 
of modulus $\leq 1$, from which it follows that $\det(I-\rho_{\leq 1}(x)D_{\leq 1})\geq 0$ (see also \cite[Theorem 4.6]{ddm05-1})
 and so the second factor in the equality above does not influence the sign of $\det(I-\rho(x)\D_\ast)$. 

\medskip

Therefore, for every $x\in F$, the following holds: $$\det(I-\rho(x)\D_\ast)\det(I-\rho_{>1}(x)D_{>1})\geq 0.$$
In particular, this inequality is true for the identity element of $F$:
$$\det(I-\D_\ast)\det(I-D_{>1})\geq 0.$$
By using both inequalities and because of Lemma \ref{expanding} (for $D_{>1}$ and $\rho_{> 1}$), 
we deduce that 
$$\det(I-\rho(x)\D_\ast)\det(I-\D_\ast)\det(\rho_{>1}(x))\geq 0,$$
which concludes this proposition.
\end{proof}

\section{Nielsen numbers and the positive part of a map}

In this section we will prove the main results of this paper. In the next theorem, we show how certain maps on a given infra-nilmanifold give rise to a specific 2-fold covering of that infra-nilmanifold, in such a way that the map under consideration  lifts to that covering. 
\begin{theorem}\label{positivepart}
Let $M=\Gamma\backslash G$ be an infra-nilmanifold modeled on a connected, simply connected nilpotent Lie group $G$, with
fundamental group the almost-Bieberbach group $\Gamma\subseteq \Aff(G)$. Let $p:\Gamma\ra F$ denote the 
projection of $\Gamma$ onto its holonomy group and denote the holonomy representation by
$\rho:F\to \GL_n(\lie)$. 
Choose an arbitrary self-map $f:M\to M$ and let $(\dd,\D)\in \aff(G)$ be an affine homotopy lift of $f$. Let 
$\rho=\rho_{\leq 1} \oplus \rho_{>1}$ be the decomposition of $\rho$ induced by $\D$.
Then 
\[ \Gamma_+ = \{ \gamma \in \Gamma\;|\; \det(\rho_{>1}(p(\gamma))) =1 \}\]
is a normal subgroup of $\Gamma$ of index 1 or 2. It follows that 
$\Gamma_+$ is also a Bieberbach group, that the corresponding infra-nilmanifold $M_+= \Gamma_+\backslash G$ is either 
equal to $M$ or a 2-fold covering of $M$. In the later case, the map $f$ lifts to a map 
$f_+:M_+\ra M_+$ which has the same affine homotopy lift $(\dd,\D)$ as $f$.
\end{theorem}
Remark: when $\rho=\rho_{\leq 1}$ (so when $\D_\ast$ has no eigenvalues of modulus $>1$), we will take $\Gamma_+=\Gamma$.
\begin{proof} We may assume that $\D_\ast$ has at least one eigenvalue of modulus $>1$, otherwise the 
theorem is trivially true. 
Note that $\Gamma_+$ is the kernel of the morphism
\[ \Gamma \ra \{-1,+1\}: \gamma \mapsto \det( \rho_{>1} (p(\gamma)))\]
and hence $\Gamma_+$ is either equal to $\Gamma$ (in case $ \det( \rho_{>1} (p(\gamma)))=1$ for all $\gamma$) or 
$[\Gamma:\Gamma_+]=2$ (in case $ \exists \gamma \in \Gamma: \det( \rho_{>1} (p(\gamma)))=-1$). In any of these two cases,
$\Gamma_+$ is still an almost--Bieberbach group and when $[\Gamma:\Gamma_+]=2$, we have that $M_+$ is a 
2-fold covering of $M$.

There is a lift $\tilde{f}: G \ra G$ of $f$ and a morphism $\varphi:F \ra F$ such that 
\[ \forall \gamma \in \Gamma:\; \varphi(\gamma)  \tilde{f}  = \tilde{f}  \gamma\]
and also 
\begin{equation}\label{voorsubiet}
\forall \gamma \in \Gamma:\; \varphi(\gamma)  (\dd, \D)   = (\dd,\D)  \gamma.
\end{equation}

We need to show that $\tilde{f}$ also induces a map on $M_+=\Gamma_+\backslash G$. 
For this, we need to prove that $\varphi(\Gamma_+) \subseteq \Gamma_+$.
As before, we can assume that 
\[ \D_\ast= \left( \begin{array}{cc}
D_{\leq 1} & \ast \\
0 & D_{>1}
\end{array}\right).\]
Let $\gamma=(a,\A)\in \Gamma_+$ and assume that $\varphi(\gamma)= (b,\B)$. As $\gamma \in \Gamma_+$, we have that 
$\det(\rho_{>1}(\A)) = 1$. Equation (\ref{voorsubiet}) implies that 
\[ \B \D = \D \A \Rightarrow \B_\ast \D_\ast = \D_\ast \A_\ast \Rightarrow 
\rho_{>1}(\B) D_{>1} = D_{>1} \rho_{>1}(\A) \]
As $\det(D_{>1})\neq 0$, this last equality implies that $\det(\rho_{>1}(\B))=\det( \rho_{>1}(\A))$, from which it follows
that $\varphi(\gamma)=(b,\B)\in \Gamma_+$.
\end{proof}

\begin{definition}
With the notations from Theorem~\ref{positivepart}, we will call $\Gamma_+$ the positive part of $\Gamma$ with respect to 
$f$. We will say that $f_+$ is the positive part of $f$.
\end{definition}

Note that for any $k\in \N$ we can take $(\dd,\D)^k$ as an affine homotopy lift of $f^k$.
Therefore, the decomposition of $\rho$ into a direct sum $\rho=\rho_{\leq 1} \oplus \rho_{>1}$ is independent 
of $k$.  It follows that the positive part $\Gamma_+$ of $\Gamma$ with respect to $f$ 
is the same as the positive part of $\Gamma$ with respect to $f^k$ for any $k\in \N$. We also have that 
$(f_+)^k=(f^k)_+$.

\medskip

The proof of the following lemma can be left to the reader.
\begin{lemma}\label{LemmaSign}
Let $D\in \R^{n\times n}$ be an arbitrary matrix. Let $p$ denote the number of real positive eigenvalues of $D$ which are 
strictly bigger than 1 and let $n$ denote the number of negative real eigenvalues of $D$ which are strictly smaller than $-1$, 
then for all $k\in \N$:
$$\left\{\begin{array}{cc}
(-1)^p \det(I-D^k)\geq 0 & \textrm{ if } $k$ \textrm{ is odd}\\
(-1)^{p+n} \det(I-D^k)\geq 0 & \textrm{ if } $k$ \textrm{ is even}.\\
\end{array}\right.$$
It follows that one of the following holds$$\forall k\in \N: \det(I-D^k)\det(I-D^{k+1})\geq 0$$or $$\forall k\in \N: \det(I-D^k)\det(I-D^{k+1})\leq 0.$$
\end{lemma}

We are now ready to show the exact relationship between the Nielsen number of (any power of) a map $f$ and the Lefschetz number of  (any power of) that map $f$ and its positive part $f_+$.

\begin{theorem}\label{NielsenLefschetz}
Let $G$ be a connected, simply connected, nilpotent Lie group, $\Gamma\subseteq \Aff(G)$ an almost--Bieberbach group, $M=\Gamma\backslash G$ the corresponding  infra-nilmanifold 
and $f:M\to M$ a map with affine homotopy lift $(\dd,D)$. 
Let $p$ denote the number of positive real eigenvalues of $\D_\ast$ which are strictly bigger than $1$ and 
let $n$ denote the number of negative real eigenvalues of $\D_\ast$ which are strictly smaller than $-1$. 
Then we can express $N(f^k)$, for $k\in \N$, in terms of $L(f^k)$ and $L(f_+^k)$, where $f_+$ is the positive part of $f$ as follows:

\renewcommand{\arraystretch}{1.7}
\begin{center}\begin{tabular}{ |c|c|c| }
\cline{2-3}
\multicolumn{1}{c}{}
 &  \multicolumn{1}{|c|}{$\Gamma=\Gamma_+$}
 & \multicolumn{1}{|c|}{$\Gamma\neq \Gamma_+$} \\
\cline{1-3}
$k$ odd & $N(f^k)=(-1)^p L(f^k)$ & $N(f^k)=(-1)^{p} (L(f_+^k)-L(f^k))$ \\[1ex]
\cline{1-3}
$k$ even & $N(f^k)=(-1)^{p+n} L(f^k)$ & $N(f^k)=(-1)^{p+n} (L(f_+^k)-L(f^k))$ \\[1ex]
\cline{1-3}
\end{tabular}
 \end{center}
\end{theorem}

\begin{proof}
Theorem \ref{LeeForm} gave us the following formulas:
\begin{equation}\label{LefschetzFormula} L(f^k)=\frac{1}{\# F}\sum_{\A \in F}\det(I-\A_\ast\D_\ast^k)\end{equation}
and
\begin{equation}\label{NielsenFormula} N(f^k)=\frac{1}{\# F}\sum_{\A \in F}|\det(I-\A_\ast\D_\ast^k)|.\end{equation}
Due to Proposition~\ref{Sign} and Theorem~\ref{positivepart}, we know that all elements of the form $\det(I-\A_\ast\D_\ast^k)$ have the same sign as $\det(I-\D_\ast^k)$ if and only if $\Gamma=\Gamma_+$. If $\Gamma\neq \Gamma_+$, on the other hand, then only half of these elements will have the same sign, since $\Gamma_+$ is an index-two-subgroup of $\Gamma$.

First, suppose $\Gamma=\Gamma_+$ and $k$ is odd. By using Lemma \ref{LemmaSign}, we find that $$|\det(I-\D_\ast^k)|=(-1)^p\det(I-\D_\ast^k)\geq 0.$$Since all terms in equation (\ref{LefschetzFormula}) have the same sign, we can replace the absolute values in equation (\ref{NielsenFormula}) by multiplying with $(-1)^p$. Hence, we get$$N(f^k)=(-1)^p L(f^k).$$If $k$ is even, a similar argument shows that $$N(f^k)=(-1)^{p+n} L(f^k).$$

Now, suppose $\Gamma\neq\Gamma_+$ and $k$ is odd. Let us denote the holonomy group of $\Gamma_+$ by $F_+$. Note that $[F:F_+]=2$. 
In an obvious way, we can rewrite equation (\ref{NielsenFormula}) as follows$$ N(f^k)=\frac{1}{\# F}\left(\sum_{\A \in F_+}|\det(I-\A_\ast\D_\ast^k)|+\sum_{\A \in F\setminus F_+}|\det(I-\A_\ast\D_\ast^k)|\right).$$
By using some of the previous arguments, this gives us$$ N(f^k)=\frac{(-1)^p}{\# F}\left(\sum_{\A \in F_+}\det(I-\A_\ast\D_\ast^k)-\sum_{\A \in F\setminus F_+}\det(I-\A_\ast\D_\ast^k)\right).$$Finally, this can be rewritten as$$ N(f^k)=(-1)^p\frac{1}{\# F}\left(-\sum_{\A \in F}\det(I-\A_\ast\D_\ast^k)\right)+(-1)^p\frac{2}{\# F}\left(\sum_{\A \in F_+}\det(I-\A_\ast\D_\ast^k)\right).$$
Since $[F:F_+]=2$, we have
$$N(f^k)=(-1)^{p}(-L(f^k)+L(f_+^k)).$$
If $k$ is even, we can deduce the following formula in a similar manner:$$N(f^k)=(-1)^{p+n} (L(f_+^k)-L(f^k)).$$
\end{proof}

Finally, we can use the results above to describe the Nielsen zeta function of $f$ in terms of the 
Lefschetz zeta functions of $f$ and $f_+$.

\begin{theorem}\label{Nielsenzeta}
Let $G$ be a connected, simply connected, nilpotent Lie group, $\Gamma\subseteq \Aff(G)$ an almost--Bieberbach group, $M=\Gamma\backslash G$ the corresponding infra-nilmanifold 
and $f:M\to M$ a map with affine homotopy lift $(\dd,D)$. 
Let $p$ denote the number of positive real eigenvalues of $\D_\ast$ which are strictly bigger than $1$ and 
let $n$ denote the number of negative real eigenvalues of $\D_\ast$ which are strictly smaller than $-1$. Then $N_f(z)$ can be expressed in terms of $L_f(z)$ and $L_{f_+}(z)$ by the corresponding entry in the following table:
\renewcommand{\arraystretch}{2.4}
\begin{center}\begin{tabular}{ |c|c|c| }
\cline{2-3}
\multicolumn{1}{l}{ }
 &  \multicolumn{1}{|c|}{$\Gamma=\Gamma_+$}
 & \multicolumn{1}{|c|}{$\Gamma\neq \Gamma_+$} \\
\cline{1-3}
$p$ even, $n$ even & $N_f(z)= L_f(z) $& $N_f(z)= \ds \frac{L_{f_+}(z)}{L_{f}(z)}$ \\[1ex]
\cline{1-3}
$p$ even, $n$ odd &$N_f(z)= \ds \frac{1}{L_{f}(-z)} $ & $N_f(z)= \ds \frac{L_{f}(-z)}{L_{f_+}(-z)}$  \\[1ex]
\cline{1-3}
$p$ odd, $n$ even & $N_f(z)= \ds \frac{1}{L_{f}(z)} $& $N_f(z)= \ds \frac{L_{f}(z)}{L_{f_+}(z)}$ \\[1ex]
\cline{1-3}
$p$ odd, $n$ odd & $N_f(z)= L_f(-z) $& $N_f(z)= \ds \frac{L_{f_+}(-z)}{L_{f}(-z)}$\\[1ex]
\cline{1-3}
\end{tabular}
 \end{center}
\end{theorem}
\begin{proof}
Let us first consider the case where $n$ is even. By Theorem \ref{NielsenLefschetz}, we find, $\forall k\in \N$, $$N(f^k)=(-1)^p L(f^k) \textrm{ or } N(f^k)=(-1)^p (L(f_+^k)-L(f^k)),$$depending on whether $\Gamma$ equals $\Gamma_+$ or not. It is then straightforward to see that$$N_f(z)=L_f(z)^{(-1)^p} \textrm{ or } N_f(z)=\left(\frac{L_{f_+}(z)}{L_f(z)}\right)^{(-1)^p}.$$Now, suppose $n$ is odd and $\Gamma=\Gamma_+$. Then $$N(f^k)=\left\{\begin{array}{cc}
(-1)^p L(f^k) & \textrm{ if } $k$ \textrm{ is odd}\\
-(-1)^{p} L(f^k) & \textrm{ if } $k$ \textrm{ is even}.\\
\end{array}\right.$$Therefore, $$N_f(z)=\exp\left(-(-1)^{p}\sum_{k=1}^{\infty}L(f^k)\frac{(-z)^k}{k}\right),$$which means that $$N_f(z)=\left(\frac{1}{L_f(-z)}\right)^{(-1)^p}.$$When $\Gamma\neq \Gamma_+$, a similar argument gives us the two remaining expressions.
\end{proof}

As an immediate consequence of the above theorem, we can conclude that the Nielsen zeta 
function for maps on infra-nilmanifolds is indeed rational.

\begin{corrolary}\label{mainresult}
Let $G$ be a connected, simply connected, nilpotent Lie group. Let $M$ be an infra-nilmanifold modeled on $G$. Choose an arbitrary continuous self-map $f:M\to M$. Let $N_f(z)$ be the Nielsen zeta function of $f$, then $N_f(z)$ is a rational function.
\end{corrolary}
\begin{proof}
This follows easily from Theorem \ref{Smale} and Theorem \ref{Nielsenzeta}.
\end{proof}

\begin{remark}
The class of infra-solvmanifolds of type $(R)$ is a class of manifolds which contains the class of infra-nilmanifolds and which still shares a lot of the good (algebraic) properties of the class of infra-nilmanifolds (see \cite{hlp10-1,ll09-1}). Although we formulated the theory in terms of infra-nilmanifolds in this paper, the reader who is familiar with the the class of infra-solvmanifolds of type $(R)$ will have noticed that all results and proofs directly generalize to this class of 
manifolds. Therefore the Nielsen zeta function for maps on infra-solvmanifolds of type $(R)$ will always be a rational function.
We have chosen to formulate everything in terms of infra--nilmanifolds because this class of manifolds is much wider known and because the 
original rationality question was formulated in terms of these manifolds.\end{remark}
\section{Some examples}
In this section, we will illustrate our results by considering maps on a 3--dimensional flat manifold, so we are considering the case $G=\R^3$. This situation is notationally much simpler than the general case, because we can identify the Lie algebra of $\R^3$ with $\R^3$ itself and so, for example, we will have that $\D_\ast=\D$, etc. On the other hand, this situation is general enough to illustrate all possibilities of the formulas above.

\medskip

Take $\{e_1,e_2,e_3\}$ as the standard basis of $\R^3$. Let $\Gamma$ be the $3$-dimensional Bieberbach group generated by the elements $(e_1,I), (e_2,I),(a,\A)$, with 
$$\A= \left(\begin{array}{ccc}
-1 & 0 & 0 \\
0& -1 & 0 \\
0& 0 & 1 
\end{array}\right) \textrm{ and } a= \left(\begin{array}{c}
 0 \\
0 \\
\frac{1}{2}
\end{array}\right).$$Note that 
$(a,\A)^2=(e_3,I)$, hence $F\cong \Z_2$ and $\Gamma\cap \R^3=\Z^3$ is a lattice of $\R^3$.

\medskip

Consider the affine map $(\dd,\D):\R^3\to \R^3$ with $$\D= \left(\begin{array}{ccc}
4 & 2 & 0 \\
-1 & 1 & 0 \\
0& 0 & 5 
\end{array}\right) \textrm{ and } \dd= \left(\begin{array}{c}
 0 \\
0 \\
0
\end{array}\right).$$

One can check that 
\[ (e_3,I)^2 (a,\A)(\dd,\D) = (\dd,\D)(a,\A)\]
from which it now easily follows that $(\dd,\D) \Gamma  \subseteq \Gamma(\dd,\D)$ and 
hence there is a morphism $\varphi:\Gamma \rightarrow \Gamma$ such that 
\[ \forall \gamma\in \Gamma:\; \varphi(\gamma) (\dd,\D) = (\dd,\D) \gamma,\]
showing that  $(\dd,\D)$ induces a map $f:\Gamma\backslash \R^3 \rightarrow \Gamma\backslash \R^3$.

\medskip

The eigenvalues of  $\D$ are $2,\;3$ and $5$.   By using the formulas from Theorem \ref{LeeForm}, we find that $$L(f^k)=\frac{(1-5^k)((1-2^k)(1-3^k)+(1+2^k)(1+3^k))}{2}=(1-5^k)(1+6^k)=1^k-5^k+6^k-30^k.$$By using the fact that $$\sum_{k=1}^{\infty}\frac{\lambda^k z^k}{k}=-\ln(1-\lambda z),$$we find that $$L_f(z)=\frac{(1-5z)(1-30z)}{(1-z)(1-6z)}.$$Because every eigenvalue of $\D$ is strictly larger than $1$, we have that $D_{>1}=\D$ and because 
$\det(\A)=1$, it follows immediately that $\Gamma=\Gamma_+$. With the same notation as above, we see that $p=3$ and $n=0$. Therefore, by Theorem \ref{NielsenLefschetz} and Theorem \ref{Nielsenzeta}, we find that $$N(f^k)=-L(f^k) \textrm{ and } N_f(z)=L_f(z)^{-1}= \frac{(1-z)(1-6z)}{(1-5z)(1-30z)}.$$

Now consider the map $g:\Gamma\backslash \R^3\to \Gamma\backslash \R^3$, induced by the affine map $(\dd',\D'):\R^3\to \R^3$ with $$\D'= \left(\begin{array}{ccc}
-2 & 8 & 0 \\
-1 & 4 & 0 \\
0& 0 & -3 
\end{array}\right) \textrm{ and } \dd'= \left(\begin{array}{c}
 0 \\
0 \\
0
\end{array}\right).$$
The fact that $(\dd',\D')$ induces a map on $\Gamma \backslash \R^3$, follows from the fact that 
\[ (e_3,I)^{-2}(a,\A)(\dd',\D') = (\dd',\D')(a,\A).\] 
Again, a straightforward calculation shows that $0$, $2$ and $-3$ are the eigenvalues of $\D'$. In a similar way as before, one can check that $$L(g^k)=1-(-3)^k \textrm{ and } L_g(z)=\frac{1+3z}{1-z}.$$Note that $\A$ and $\D'$ are simultaneously diagonalizable. Using this diagonalization, we have that 
\[ D_{>1}=\left(\begin{array}{cc}
2 & 0 \\ 0 & -3
\end{array}\right)\mbox{ and }
\rho_{>1} (\A)=\left( \begin{array}{cc}
-1 & 0 \\0 & 1
\end{array}\right) .\] 
Since $\det(\rho_{>1}(\A))=-1$ we know $\Gamma\neq \Gamma_+$. In fact $\Gamma_+=\Gamma\cap \R^3$, from which it follows that $g_+$ is a map on the $3$-dimensional torus $T^3$, such that $$L(g_+^k)=\det(I-\D)=(1-2^k)(1-(-3)^k)=1^k-2^k-(-3)^k+(-6)^k.$$We deduce that $$L_{g_+}(z)=\frac{(1-2z)(1+3z)}{(1-z)(1+6z)}.$$Because $p=n=1$ and because of Theorem \ref{Nielsenzeta}, we find  $$N_g(z)=\frac{L_{g_+}(-z)}{L_{g}(-z)}=\frac{1+2z}{1-6z}.$$Note that this expression for the Nielsen zeta function allows us to say that $$N(g^k)=6^k-(-2)^k,$$which could also be computed by using Theorem \ref{NielsenLefschetz}.

\end{document}